\documentclass[12pt]{article}
\usepackage{graphicx}
\usepackage{amsmath}
\usepackage{amsthm}
\usepackage{latexsym}
\usepackage{amsfonts}
\usepackage{etoolbox}
\usepackage{lipsum}
\let\bbordermatrix\bordermatrix
\patchcmd{\bbordermatrix}{8.75}{4.75}{}{}
\patchcmd{\bbordermatrix}{\left(}{\left[}{}{}
\patchcmd{\bbordermatrix}{\right)}{\right]}{}{}
\usepackage{multirow}
\usepackage{lipsum}
\newcommand\blfootnote[1]{%
  \begingroup
  \renewcommand\thefootnote{}\footnote{#1}%
  \addtocounter{footnote}{-1}%
  \endgroup
}
\newtheorem{theorem}{Theorem}[section]
\newtheorem{lemma}[theorem]{Lemma}
\newtheorem{corollary}{Corollary}[theorem]

\theoremstyle{definition}

\newtheorem{rem}[theorem]{Remark}
\newtheorem{conjecture}[theorem]{Conjecture}

\usepackage[T1]{fontenc}
\usepackage[utf8]{inputenc}
\usepackage{authblk}

\title {\bf Inverse sum indeg energy of graphs}
\author {Sumaira Hafeez}
\author {Rashid Farooq\thanks{Corresponding author.} }
\affil{School of Natural Sciences,
	National University of Sciences and Technology,\\
	H-12 Islamabad, Pakistan }

\date{}
\begin{document}
\maketitle
\date{}
\blfootnote{\raggedright Email addresses:  sumaira.hafeez123@gmail.com (S. Hafeez), farook.ra@gmail.com (R. Farooq).}
\begin{abstract}
Suppose $G$ is an $n$-vertex simple graph with vertex set $\{v_1,\dots,v_n\}$ and $d_i$, $i=1,\dots,n$, is the degree of vertex $v_i$ in $G$. The ISI matrix $S(G)= [s_{ij}]_{n\times n}$ of $G$ is defined by $s_{ij}= \frac{d_i d_j}{d_i+d_j}$ if the vertices $v_i$ and $v_j$ are adjacent and $s_{ij}=0$ otherwise. The $S$-eigenvalues of $G$ are the eigenvalues of its ISI matrix $S(G)$. Recently the notion of inverse sum indeg (henceforth, ISI) energy of graphs is introduced and is defined by $\sum\limits_{i=1}^{n}|\tau_i|$, where $\tau_i$ are the $S$-eigenvalues. We give ISI energy formula of some graph classes. We also obtain some  bounds for ISI energy of graphs.
\end{abstract}
\begin{quote}
\end{quote}

\begin{quote}
\end{quote}
\section{Introduction}
A graph $G$ is a pair $G =(V(G), E(G))$, where $V(G)$ denotes the vertex set $\{v_1,\dots,v_n\}$ and $E(G)$ denotes the edge set of $G$. The degree $d_i$ of a vertex $v_i$ is the number of edges incident on it. If vertices $v_i$ and $v_j$ are adjacent, we denote it by $v_i\sim v_j$. If vertices $v_i$ and $v_j$ are not adjacent, we denote it by $v_i\not\sim v_j$. An $n$-vertex path $P_n$, $ (n \geq 1)$, is a graph with vertex set $\{v_1,\dots v_n\}$ and edge set $\{v_jv_{j+1}|j = 1,2 \dots, n-1 \}$. An $n$-vertex cycle $C_n$$(n\geq 3)$ is a graph with vertex set $\{v_1, \dots, v_n\}$ and edge set $\{v_jv_{j+1}| j = 1, 2, \dots, n -1\} \cup \{v_nv_1\}$. The star graph $S_n$ of order $n$ is isomorphic to $K_{1,n-1}$. By $\overline{G}$, we denote the complement of $G$.

The adjacency matrix $A(G)= [a_{ij}]_{n\times n }$ of an $n$-vertex graph $G$ is defined as
\[
a_{ij} =
\left\{
\begin{array}{ll}
1 &\mbox{if $ v_i \sim v_j$,}\\
0 &\mbox{otherwise.}
\end{array}
\right.
\]   
The $A$-characteristic polynomial of $G$ is the polynomial
\begin{eqnarray*}
\Phi(G,\lambda)&=&\det(A(G) - \lambda I_n)\\
&=& \lambda^n+ \sum_{i=1}^{n} a_{i} \lambda^{n-i}.
\end{eqnarray*} 
where $I_n$ is the diagonal matrix of order $n$ with diagonal entries equal to 1. The $A$-eigenvalues of $G$ are the $A$-eigenvalues of $A(G)$. The spectrum of $G$, denoted by ${\rm spec}_A(G)$, is the set of $A$-eigenvalues of $G$ together with their multiplicities.
	
The inverse sum indeg, (henceforth ISI) index, was studied in \cite{DV}. The ISI index is defined as
\begin{equation*}
\emph{\emph{ISI}}(G) = \sum\limits_{i \sim j}\frac{d_id_j}{d_i+d_j}.
\end{equation*}

Zangi et al. \cite {Z2018} defined the ISI matrix $S(G)= [s_{ij}]_{n\times n }$ of an $n$-vertex graph $G$ as:
\[
s_{ij} =
\left\{
\begin{array}{ll}
\frac{d_i d_j}{d_i+d_j} &\mbox{if $ v_i \sim v_j$,}\\
0 &\mbox{otherwise.}
\end{array}
\right.
\]   
The $S$-characteristic polynomial of $G$ is given by:
\begin{eqnarray}\nonumber
\Phi_S(G,\lambda)&=&\det(S(G) - \lambda I_n)\\\label{e1}
&=& \lambda^n+ \sum_{i=1}^{n} b_{i} \lambda^{n-i},
\end{eqnarray} 
where $I_n$ is the diagonal matrix of order $n$ with diagonal entries equal to 1. The $S$-eigenvalues of $G$ are the $S$-eigenvalues of $S(G)$. The $S$-spectrum ${\rm spec}_{S}(G)$ of $G$, is the set of $S$-eigenvalues of $G$ together with their multiplicities. Since ISI matrix of graph is symmetric and real, therefore its eigenvalues are real. If $G$ is an $n$-vertex graph with distinct $S$-eigenvalues $\tau_1, \tau_2, \dots, \tau_k$ and if their respective
multiplicities are $p_1, p_2, \dots ,p_k$, we write the $S$-spectrum of $G$ as $\rm spec_{\bf{S}}(G) = \{\tau_1^{(p_1)}, \tau_2^{(p_2)}, \dots, \tau_k^{(p_k)}\}$.

In $1978$, the energy of a simple graph is defined by Gutman \cite{IG} as $E(G) = \sum\limits_{i=1}^{n}|\lambda_i|$. Many results on the graph energy can be found in literature. The concept of Randic energy is given by Bozkurt et al. \cite{BGG,BGC}. In 2014, Gutman et al. \cite{GFB} gave some of the properties of Randić matrix and Randić energy. Sedlar et al. \cite{JSS} study the properties ISI index and finds extremal values of ISI index for some classes of graphs. Pattabiraman \cite{KP} gave some extremal bounds on ISI index. In 2018, Das et al. \cite{DGM} summarized different types of energies of graphs introduced by many authors. Das et al. \cite{DGM} find some of the lower and upper bounds for these energies of graphs. For recent results on different types of energies of graphs, one can study \cite{ CFK, DG, E, MMG, PRC, RJG, RPH, RS, ZT}.     

Zangi et al. \cite{Z2018} introduce the concept of ISI energy of graphs. In this paper we obtain ISI energy formula of some well-known graphs. Upper and lower bounds are established. Finally, we give integral representation for ISI energy of graphs. 

\section{Inverse sum indeg energy}
Let $\lambda_1, \dots, \lambda_n$ be $A$-eigenvalues of an $n$-vertex graph $G$. Then Gutman \cite{IG} defined the energy of $G$ as $ E(G) = \sum\limits_{i=1}^{n} |\lambda_i|$. 

Let $\tau_1, \dots, \tau_n$ be the $S$-eigenvalues of $G$. Then Zangi et al. \cite{Z2018} define ISI energy of $G$ as
\begin{equation}
E_{{\rm ISI}}(G) = \sum\limits_{i=1}^{n} |\tau_i|.
\end{equation}
For convenience , we define some notations. We denote determinant of $S(G)$ by $\det (S(G))$. Let
\begin{equation*}
Q= 2 \sum\limits_{1\leq i<j \leq n} \bigg(\frac{d_i d_j}{d_i+d_j}\bigg)^2, ~~~ \Theta =\det(S(G)).
\end{equation*}
The trace of the matrix $S(G) = [s_{ij}]_{n \times n}$ is defined by $\sum\limits_{i=1}^{n} s_{ii}$ and is denoted by $tr(S(G))$.
Zangi et al. \cite{Z2018} prove the following lemma.
\begin{lemma}[Zongi et al. \cite{Z2018}]\label{trace}
Let $G$ be an $n$-vertex graph and let $\tau_1, \dots, \tau_n$ be its $S$-eigenvalues. Then
\begin{itemize}
	\item [$(1)$] $\sum\limits_{i=1}^{n} \tau_i =0$,
	\item[$(2)$] $tr(S^2(G)) = \sum\limits_{i=1}^{n} \tau_i^2 =Q.$
\end{itemize}	
\end{lemma} 
\begin{theorem}\label{trace2}
Let $G$ be an $n$-vertex simple and connected graph and let $m$ be the number of edges in $G$. Then
\begin{equation*}
tr(S^2(G)) \leq tr(S^2(K_n)),
\end{equation*}
where the equality holds if and only if $G\cong K_n$.
\end{theorem}
\begin{proof}
First let $G \not \cong K_n$. Then $d_i \leq n-1$ for every vertex $v_i$ of $G$, $i=1,\dots,n$. Therefore   
	\begin{equation*}
	\frac{d_i d_j}{d_i +d_j} =\frac{1}{\frac{1}{d_i}+\frac{1}{d_j}} \leq \frac{1}{\frac{1}{n-1}+\frac{1}{n-1}} = \frac{n-1}{2},
	\end{equation*}
	\begin{equation*}
	tr(S^2(G))= 2 \sum\limits_{1\leq i<j \leq n} \bigg(\frac{d_i d_j}{d_i+d_j}\bigg)^2 \leq 2 m \frac{(n-1)^2}{4} = \frac{m(n-1)^2}{2}.
	\end{equation*}
	As $G \not \cong K_n$, it holds that $m <\frac{n(n-1)}{2}$. Consequently
	\begin{eqnarray*}
		tr(S^2(G)) \leq \frac{m(n-1)^2}{2} &<& \frac{n(n-1)}{2} \times \frac{(n-1)^2}{2} \\
		&=& \frac{n(n-1)^3}{4}. 
	\end{eqnarray*}
Now let $G \cong K_n$. Lemma \ref{trace} implies
\begin{eqnarray*}
tr (S^2(K_n))&=& 2\left[\frac{n(n-1)}{2} \times\frac{(n-1)^2}{4}\right]\\
&=& \frac{n(n-1)^3}{4}.
\end{eqnarray*}

Hence $tr(S^2(G)) \leq tr(S^2(K_n))$.
\end{proof}
 Suppose $G_1$ and $G_2$ are two graphs with disjoint vertex sets. Then the graph union $G_1 \cup G_2$ is a pair $G_1 \cup G_2 =(V(G_1\cup G_2), E(G_1 \cup G_2))=(V(G_1) \cup V(G_2), E(G_1) \cup E(G_2))$. The degree of a vertex $v$ of $G_1 \cup G_2$ is equal to the degree of the vertex $v$ in the
component $G_i$, $i = 1,2$, that contains it. A square diagonal matrix whose diagonal elements are square matrices and the non-diagonal elements are 0 is called a block diagonal matrix. 

Next theorem gives the relation between ISI energy of a graph and its components.
\begin{theorem}\label{component}
Suppose $G_1, G_2, \dots, G_s$ are the components of a graph $G$. Then $E_{{\rm ISI}}(G) =\sum\limits_{i=1}^{s} E_{{\rm ISI}}(G_i)$.
\end{theorem}
\begin{proof}
Since $G_1, G_2, \dots, G_s$ are the components of $G$, we can write $G= G_1\cup G_2 \cup \dots \cup G_s$. Then $S(G)$ is a block diagonal matrix with diagonal elements $S(G_1), S(G_2), \dots S(G_s)$. Therefore 
\begin{equation*}
{\rm spec}_S(G) = {\rm spec}_S(G_1)\cup {\rm spec}_S(G_2) \cup \dots \cup{\rm spec}_S(G_s).
\end{equation*}
Hence 
\begin{equation*}
E_{{\rm ISI}}(G) =\sum\limits_{i=1}^{k} E_{{\rm ISI}}(G_i).
\end{equation*}
\end{proof}
Following result follows directly from ISI matrix of $\overline{K}_n$.
\begin{lemma}
Suppose $G$ is an $n$-vertex graph. Then $E_{{\rm ISI}}(G) = 0$ if and only if $G\cong \overline{K}_n$.
\end{lemma}
We now show that the ISI energy of a non-trivial graph, if it is an integer, must be an even positive integer.
\begin{theorem}
If $G\not\cong \overline{K}_n$ and the {\rm ISI} energy of a graph $G$ is an integer then it must be an even positive integer.
\end{theorem}
\begin{proof}
Let $\tau_1,\dots,\tau_n$ be $S$-eigenvalues of $G$ and with no loss of generality, assume that $\tau_1,\dots,\tau_s$ are positive and $\tau_{s+1},\dots,\tau_{n}$ are non-negative. From Lemma \ref{trace}, we have
\begin{eqnarray*}
		\sum\limits_{i=1}^{s} \tau_i + 	\sum\limits_{i=s+1}^{n} \tau_i=\sum\limits_{i=1}^{n} \tau_i&=&0
\end{eqnarray*}
This gives
\begin{equation*}
\sum\limits_{i=1}^{s} \tau_i=-\sum\limits_{i=s+1}^{n} \tau_i.
\end{equation*}
Now
\begin{eqnarray*}
E_{{\rm ISI}}(G) &=& \sum\limits_{i=1}^{n} |\tau_i|\\
&=& \sum\limits_{i=1}^{s} \tau_i + \sum\limits_{i=s+1}^{n} (-\tau_i)\\
&=&2\sum\limits_{i=1}^{s} \tau_i.
\end{eqnarray*}
Therefore ISI energy of $G$ is an even integer.
\end{proof}
The distance between $v$ and $u$ of $G$ is the length of the shortest path between them. The maximum distance between a vertex $v$ to all other vertices of $G$ is called the eccentricity of $v$. The diameter of $G$ is the maximum eccentricity of any vertex in $G$. A matrix $M$ is irreducible if the digraph associated with $M$ is strongly connected. A matrix is non-negative if its all entries are non-negative.

In the following two results, we determine some properties of the $S$-eigenvalues. The idea of proof is taken from proof of Lemma $1.1$ \cite{K2017}
\begin{lemma}\label{lem3}
Let $G$ be an $n$-vertex simple and connected graph, $n \geq 2$, with non-incresing $S$-eigenvalues $\tau_1\geq \tau_2\geq \dots \geq \tau_n$. If $G$ has diameter atleast $3$, then $\tau_1 > \tau_2 >0$. 
\end{lemma}
\begin{proof}
Since the graph $G$ is connected therefore $S(G)$ is an irreducible non-negative square matrix of order $n$. By Perron-Frobenius theorem, we have $ \tau_1 > \tau_2$. Since $G$ has diameter at least
3, $P_4$ is the subgraph of $G$. Therefore we have $\tau_2(G) \geq \tau_2(P_4)= \frac{4}{3} >0$, where $\tau_2(G)$ is the second largest $S$-eigenvalue of $G$ and $\tau_2(P_4)$ is the second largest $S$-eigenvalue of $P_4$ . Hence $\tau_1 > \tau_2 >0$.
\end{proof}
\begin{lemma}[Brouwer and Haemers \cite{BH}]\label{lem2}
Let $G$ be a connected graph with greatest eigenvalue $\lambda_1$. Then $-\lambda_1$ is an eigenvalue of $G$ if and only if $G$ is bipartite.
\end{lemma}
\begin{theorem}
Suppose $G$ is an $n$-vertex graph, $n \geq2$, with $S$-eigenvalues $\tau_1, \dots,\tau_n$ and let its $A$-spectrum and $S$-spectrum are symmetric about the origin. Then $|\tau_1|=|\tau_2|=\dots=|\tau_m| >0$ $(m \geq2)$ and the remaining $S$-eigenvalues are zero (if exist) if and only if $G \cong \bigcup \limits_{j=1}^{p} K_{r,s}$, where $ p(r+s)= n$ and one of the $r$ or $s$ is greater than 1.	
\end{theorem}
\begin{proof}
First assume that
\begin{equation} \label{e2}
|\tau_1|=|\tau_2|=\dots=|\tau_m| >0 ~(m \geq2)
\end{equation} 
and the remaining $S$-eigenvalues are zero (if exist). Then each component of $G$ has atmost three distinct $S$-eigenvalues. Let $H$ be a component of $G$. From \eqref{e2} and Lemma \ref{lem2}, we see that $H$ is bipartite. If $H$  is not a complete bipartite graph, then the diameter of $H$ is at least 3. Therefore by Lemma \ref{lem3} and \eqref{e2}, we get a contradiction. Hence $H$ is a complete bipartite graph. As $H$ is arbitrary component of $G$, therefore $G \cong \bigcup \limits_{j=1}^{p} K_{r,s}$, where $ p(r+s)= n$. 

The converse statement is easy to prove. 
\end{proof}
\section{ISI energy of some graphs}
In this section, we prove ISI energy formulae for some classes of graphs. 

The $A$-spectrum of $K_n$ and $K_{m,n}$ is given by
\begin{eqnarray*}
{\rm spec}_A(K_n) &=& \{(-1)^{n-1}, (n-1)\},\\
{\rm spec}_A(K_{m,n})&=& \{(0)^{m+n-2}, \pm \sqrt{mn}\}
\end{eqnarray*}

Bhat and Pirzada \cite{MS} gave the following energy formulae for cycle $C_n$ of order $n$:
\begin{equation}\label{posi1}
E(C_n)=
\left\{
\begin{array}{ll}
	4~\cot \frac{\pi}{n} & \mbox{if $n\equiv 0 (\bmod 4)$}\vspace{.2cm}\\
	4~\csc \frac{\pi}{n} & \mbox{if $n\equiv 2 (\bmod 4)$}\vspace{.2cm}\\
	2~\csc \frac{\pi}{2n} & \mbox{if $n\equiv 1 (\bmod 2)$.}
\end{array}
\right.
\end{equation}

\begin{theorem}\label{cycle}
$E_{{\rm ISI}}(C_n) = E(C_n)$
\end{theorem}
\begin{proof}
Since degree of every vertex in $C_n$ is $2$, therefore for any $v_i, v_j \in V(C_n)$ with $v_i \sim v_j$, we have
\begin{eqnarray*}
s_{ij}&=& \frac{d_i d_j}{d_i +d_j}\\
&=& \frac{2 \times 2}{2+2} =1.
\end{eqnarray*}
If $v_i \not\sim v_j$ then $s_{ij}=0$. Hence $ S(C_n) = A(C_n)$. Consequently ${\rm spec}_S(G)= {\rm spec}_A(G)$ and $E_{{\rm ISI}}(C_n) = E(C_n)$.
\end{proof}
\begin{theorem}\label{bipartite}
$E_{{\rm ISI}}(K_{m,n})=\frac{2(mn)^{\frac{3}{2}}}{m+n}$.
\end{theorem}
\begin{proof}
	Let $B$ be an $m \times n$ matrix and $C$ be an $n \times m$ matrix, where all entries of $B$ and $C$ are equal to $\frac{mn}{m+n}$. Let $O$ be a zero matrix of order $m\times m$ and $O'$ be a zero matrix of order $n\times n$. Then 
	\begin{equation*}
	S(K_{m,n})=
	\begin{bmatrix} 
	O & B \\
	C& O' 
	\end{bmatrix}.
	\end{equation*}
	That is,
	\begin{equation*}
	S(K_{m,n}) = \frac{mn}{m+n}~ A(K_{m,n}).
	\end{equation*} 
	Hence 
	\begin{equation*}
	{\rm spec}_S(K_{m,n}) = \left\{\frac{(mn)^{\frac{3}{2}}}{m+n},~ 0^{(m+n-2)}, ~ -\frac{(mn)^{\frac{3}{2}}}{m+n}\right\}.
	\end{equation*}
	Therefore
	\begin{eqnarray*}
		E_{{\rm ISI}}(K_{m,n}) &=& \sum\limits_{j=1}^{m+n} |\tau_j|\\
		&=& \bigg|\frac{(mn)^{\frac{3}{2}}}{m+n}\bigg|+ \bigg|-\frac{(mn)^{\frac{3}{2}}}{m+n}\bigg|\\
		&=& \frac{2(mn)^{\frac{3}{2}}}{m+n}.
	\end{eqnarray*}
The proof is complete.
\end{proof}
Following corollary is an easy consequence of Theorem \ref{bipartite} .
\begin{corollary}\label{cor1}
 $E_{{\rm ISI}} (S_n) = \frac{2(n-1)^{\frac{3}{2}}}{n}$.	
\end{corollary}
\begin{theorem}\label{complete}
$E_{{\rm ISI}}(K_n)=(n-1)^2$.
\end{theorem}
\begin{proof}
Since each vertex of $K_n$ has degree $n-1$, for $v_i, v_j \in V(K_n)$ with $i\neq j$, we have  
\begin{eqnarray*}
	s_{ij}&=& \frac{d_i d_j}{d_i +d_j}\\
	&=& \frac{(n-1)^2}{(n-1)+(n-1)} =\frac{n-1}{2}. 
\end{eqnarray*}
Hence
\begin{equation*}
S(K_n) = \frac{n-1}{2}~ A(K_n).
\end{equation*}
Consequently 
\begin{equation*}
{\rm spec}_S(K_n) = \left\{\frac{(1-n)}{2}^{(n-1)},~ \frac{(n-1)^2}{2}\right\}.
\end{equation*}
Now
\begin{eqnarray*}
E_{{\rm ISI}}(K_n) &=& \sum\limits_{j=1}^{n} |\tau_j|\\
&=& \sum\limits_{j=1}^{n-1} \bigg|\frac{(1-n)}{2}\bigg|+\bigg|\frac{(n-1)^2}{2}\bigg|\\
&=& (n-1)^2.
\end{eqnarray*}
\end{proof}
\begin{rem}
By Theorem \ref{component} and Theorem \ref{complete}, it is easily seen that $E_{{\rm ISI}}(\overline{K}_{m,n}) = m^2 +n^2 -2(m+n-1)$.
\end{rem}
A graph whose every vertex has equal degree is called a regular graph. A graph whose every vertex has degree $k$ is called a $k$-regular graph.  Zangi et al. \cite{Z2018} prove the following result.
\begin{theorem}[Zangi et al. \cite{Z2018}]\label{regular}
Suppose $G$ is an $n$-vertex $k$-regular graph. Then $E_{{\rm ISI}}(G) = \frac{k}{2}~ E(G)$.   
\end{theorem}

\section{Bounds and integral representation for ISI energy}
In this section, we give some bounds for the ISI energy of graphs.

Let $B$ is a matrix of order $n\times n$ such that $b_{ij} =0$ if $ v_i \not\sim v_j$ and $b_{ij} = \mathcal{F} (d_i,d_j)$ if $v_i \sim v_j$, where $ \mathcal{F}$ is the function with the property $\mathcal{F}(y,z) = \mathcal{F}(z,y)$.  
Das et al. \cite{DGM} prove the following theorem for eigenvalues of degree based energies of graphs.
\begin{theorem}[Das et al. \cite{DGM}]\label{eigen}
For the eigenvalues $f_1 \geq f_2 \geq \dots \geq f_n$ of a matrix $M$, the followimg inequalities hold.
\begin{eqnarray*}
\sqrt{\frac{tr(B^2)}{n(n-1)}} &\leq& f_1 \leq \sqrt{\frac{(n-1) ~tr(B^2)}{n}},\\
-\sqrt{\frac{(n-1) ~tr(B^2)}{n}} &\leq & f_n \leq  -\sqrt{\frac{ tr(B^2)}{n(n-1)}},\\
-\sqrt{\frac{(k-1)~ tr(B^2)}{n(n-k+1)}} & \leq& f_k \leq \sqrt{\frac{(n-k) ~tr(B^2)}{kn}}, 
\end{eqnarray*}
for $k=2,\dots,n-1$.
\end{theorem}
The following result is obtained by using Theorem \ref{eigen},.
\begin{theorem}\label{eigenvalues}
For the eigenvalues $\tau_1 \geq \tau_2 \geq \dots \geq \tau_n$ of $S(G)$, the following inequalities hold.
\begin{eqnarray*}
	\sqrt{\frac{Q}{n(n-1)}} &\leq& \tau_1 \leq \sqrt{\frac{(n-1) ~Q}{n}},\\
	-\sqrt{\frac{(n-1)~ Q}{n}} &\leq & \tau_n \leq  -\sqrt{\frac{ Q}{n(n-1)}},\\
	-\sqrt{\frac{(k-1)~ Q}{n(n-k+1)}} & \leq& \tau_k \leq \sqrt{\frac{(n-k)~ Q}{kn}}, 
\end{eqnarray*}
for $k=2,\dots,n$.
\end{theorem}

Using Theorem \ref{trace2} and Theorem \ref{eigenvalues}, we get the following result for an $n$-vertex connected graph $G$.
\begin{theorem}
For $S$-eigenvalues $\tau_1 \geq \tau_2 \geq \dots \geq \tau_n$ of a connected graph $G$, the following inequalities hold.
\begin{eqnarray*}
	\sqrt{\frac{Q}{n(n-1)}} &\leq& \tau_1 \leq  E_{{\rm ISI}}(K_n),\\
	-E_{{\rm ISI}}(K_n) &\leq & \tau_n \leq  -\sqrt{\frac{ Q}{n(n-1)}},\\
	-\sqrt{\frac{(k-1)~ (n-1)^3}{4(n-k+1)}} & \leq& \tau_k \leq \sqrt{\frac{(n-k)~ (n-1)^3}{4k}}, 
\end{eqnarray*}
for $k=2,\dots,n-1$.
\end{theorem} 
In next theorem, we find bounds for ISI energy in terms of trace of matrix $S^2(G)$ and determinant of $S(G)$.
\begin{theorem}
Let $G$ be an $n$-vertex simple graph, $n \geq 2$. Then
\begin{equation*}
n |\Theta|^{\frac{2}{n}}\leq E_{{\rm ISI}}(G) \leq \sqrt{n Q},
\end{equation*}
where $\Theta = \det S(G)$
\end{theorem}
\begin{proof}
As we know that arithmetic mean is always less than quadratic mean, therefore
\begin{eqnarray*}
E_{{\rm ISI}}(G) &=& \sum\limits_{i=1}^{n} |\tau_i|\\
& \leq&  \sqrt{n~\sum\limits_{j=1}^{n} |\tau_i|^2}\\
&=& \sqrt{n~\sum\limits_{i=1}^{n} \tau_j^2} = \sqrt{n~Q}.
\end{eqnarray*}
Arithmetic-quadratic mean inequality gives,  
\begin{eqnarray*}
(E_{{\rm ISI}}(G))^2 &=& (\sum\limits_{i=1}^{n} |\tau_i|)^2\\
& \geq & \sum\limits_{i=1}^{n} |\tau_i|^2\\
& \geq & n~ (\prod\limits_{i=1}^{n} |\tau_i|)^{\frac{2}{n}} = n |\Theta|^{\frac{2}{n}}.
\end{eqnarray*}
The proof is complete.
\end{proof}
Now we have the following theorem. The proof is same as the proof of Theorem $3$ \cite{DGM} and is thus excluded.
\begin{theorem}\label{theo1}
Let $G$ be a simple $n$-vertex graph with $n \geq 2$ vertices. Then
\begin{equation*}
\sqrt{Q+ n(n-1)~|\Theta|^{\frac{2}{n}}} \leq E_{{\rm ISI}}(G) \leq \sqrt{(n-1) ~Q +n |\Theta|^{\frac{2}{n}}}.
\end{equation*}
\end{theorem}
In Theorem $4.8$, we obtain bounds for ISI energy in terms of number of edges, minimum and maximum degrees of a simple graph.
\begin{theorem}
Suppose $G$ is an $n$-vertex simple graph with $m$ edges, minimum degree $\delta$ and maximum degree $\Delta$. Then 
\begin{equation*}
\sqrt{\frac{m\delta^2}{2}+ n(n-1)~|\Theta|^{\frac{2}{n}}} \leq E_{{\rm ISI}}(G) \leq \sqrt{m(n-1) ~\frac{\Delta^2}{2} +n |\Theta|^{\frac{2}{n}}}.
\end{equation*} 
\end{theorem}
\begin{proof}
For each vertex $v_i$ of $G$, $\delta \leq d_i \leq \Delta$, $i=1,2,\dots n$. Using this fact, we get
\begin{equation*}
\frac{1}{\frac{1}{d_i}+\frac{1}{d_j}} \leq \frac{1}{\frac{1}{\Delta}+\frac{1}{\Delta}} = \frac{\Delta}{2},
\end{equation*}
\begin{equation*}
\frac{1}{\frac{1}{d_i}+\frac{1}{d_j}} \geq \frac{1}{\frac{1}{\delta}+\frac{1}{\delta}} = \frac{\delta}{2}.
\end{equation*}
Hence
\begin{equation*}
Q= 2 \sum\limits_{1\leq i<j \leq n} \bigg(\frac{d_i d_j}{d_i+d_j}\bigg)^2 \leq 2 m \frac{\Delta^2}{4} = \frac{m\Delta^2}{2},
\end{equation*}
\begin{equation*}
Q= 2 \sum\limits_{1\leq i<j \leq n} \bigg(\frac{d_i d_j}{d_i+d_j}\bigg)^2 \geq 2 m \frac{\delta^2}{4} = \frac{m\delta^2}{2}.
\end{equation*}
Now using Theorem \ref{theo1}, we obtain the desired result.
\end{proof}
Coulson \cite{Coulson} prove the following integral representation of energy of graphs.
\begin{theorem}[Coulson \cite{Coulson}]\label{coul}
Let $G$ be an $n$-vertex simple graph, then 
\begin{equation*}
	E(G) =  \frac{1}{\pi}\:\:\int_{-\infty}^{\infty} \bigg(n- \frac{\dot{\iota}\lambda\:\Phi'(G,\dot{\iota}\lambda)}{\Phi(G,\dot{\iota}\lambda)}\bigg) d\lambda,
\end{equation*}
where $\Phi'(G,\lambda) = \frac{d}{d\lambda} \Phi(G,\lambda)$ and $\dot{\iota} = \sqrt{-1}$.
\end{theorem}
Next theorem is an analogue of Theorem \ref{coul}.
\begin{theorem}\label{coul1}
Let $G$ be a simple graph of order $n$. Then 
\begin{equation*}
E_{{\rm ISI}}(G) =  \frac{1}{\pi}\:\:\int_{-\infty}^{\infty} \bigg(n- \frac{\dot{\iota}\lambda\:\Phi'_S(G,\dot{\iota}\lambda)}{\Phi_S(G,\dot{\iota}\lambda)}\bigg) d\lambda,
\end{equation*}
where $\Phi'_S(G,\lambda) = \frac{d}{d\lambda} \Phi_S(G,\lambda)$ and $\dot{\iota} = \sqrt{-1}$.
\end{theorem}
\begin{corollary}
If $G$ is an $n$-vertex graph then
\begin{equation*}
E_{{\rm ISI}}(G) =  \frac{1}{\pi}\:\:\int_{-\infty}^{\infty} \frac{1}{\lambda^2} \ln\bigg( \lambda^n \Phi_S(G,\frac{\dot{\iota}}{\lambda})\bigg).
\end{equation*}
\end{corollary} 

The following result is similar to the graph energy.
\begin{theorem}
Let $G$ be an $n$-vertex graph with $S$-characteristic polynomial $\Phi_S(G,\lambda)= \lambda^n+ \sum_{i=1}^{n} b_{i} \lambda^{n-i}.$ Then
\begin{equation*}
	E_{{\rm ISI}}(G) =  \frac{1}{2\pi}\:\:\int_{-\infty}^{\infty} \frac{1}{\lambda^2}\:\:log\:\left[\:(\:\sum\limits_{i=0}^{\lfloor\frac{n}{2}\rfloor} (-1)^i b_{2i}(G) \lambda^{2i}\:)^2 + (\:\sum\limits_{i=0}^{\lfloor\frac{n}{2}\rfloor}(-1)^i b_{2i+1}(G) \lambda^{2i+1}\:)^2 \:\right]\: d\lambda.
	\end{equation*}
\end{theorem}
The following result is based on our numerical testing. The application of Coulson-type integral expression for proving the conjecture was (so far) not successful.
\begin{conjecture}
Among all $n$-vertex tress, the tree with minimal ISI energy is $S_n$ and the tree with maximal ISI energy is $P_n$, where $n\geq2$.	
\end{conjecture}
\section{S-Equienergetic graphs}
Two graphs with same $S$-spectrum are said to be $S$-cospectral, otherwise $S$-noncospectral. Two S-equienergetic graphs of same order are the graphs which have same ISI energy. Two isomorphic graphs are always $S$-cospectral and thus are S-equienergetic. We construct few classes of $S$-noncospectral $S$-equienergetic graphs.

The line graph, denoted by $L(G)$, of $G$, is the graph with $V(L(G)) =E(G)$ and two vertices of $L(G)$ are connected by an edge if edges incident on it are adjacent in $G$.

Let $G$ be $k$-regular $n$-vertex graph. Let $L_1(G) = L(G)$, $L_i(G) =
L(L_{i-1}(G))$, $i=1,2,\dots$, be the iterated line graphs of $G$. Ramane et al. \cite{RW} prove the following energy formula for $L_2(G)$.
\begin{equation}\label{s1}
E(L_2(G)) = 2nk(k-2).
\end{equation} 
\begin{theorem}
Suppose $G_1$ and $G_2$ are two $k$-regular $n$-vertex $A$-noncospectral graphs. Then $L_2(G_1)$ and $L_2(G_2)$ are $S$-noncospectral $S$-equienergetic graphs.
\end{theorem}
\begin{proof}
If $G$ is an $n$-vertex $k$-regular then $L_2(G)$ is $\frac{1}{2} nk(k-1)$-vertex $(4k-6)$-regular graph. By Theorem \ref{regular} and \eqref{s1}, we get 
\begin{eqnarray*}
E_{{\rm ISI}}(L_2(G)) &=& (2k-3)~ E(L_2(G))\\
&=& 2nk(2k-3)(k-2).
\end{eqnarray*}
Hence $E_{{\rm ISI}}(L_2(G_1)) = E_{{\rm ISI}}(L_2(G_2))$. \vspace{.2cm}

Since $S(L_2(G))=(2k-3) A(L_2(G))$ and $L_2(G_1)$ and, $L_2(G_2)$ are $A$-noncospectral graphs, therefore $L_2(G_1)$ and $L_2(G_2)$ are also $S$-noncospectral graphs.  
\end{proof} 
\begin{corollary}
Suppose $G_1$ and $G_2$ are two $k$-regular $n$-vertex $A$-noncospectral graphs. Then for any $m \geq 2$, $L_m(G_1)$ and $L_m(G_2)$ are $S$-noncospectral $S$-equienergetic.
\end{corollary}
\begin{theorem}
Suppose $G_1$ and $G_2$ are two $n$-vertex $S$-noncospectral $S$-equienergetic graphs. Then $G_1 \cup \overline{K}_r$ and $G_2 \cup \overline{K}_r$ are $S$-noncospectral $S$-equienergetic.
\end{theorem}
\begin{proof}
By Theorem \ref{component}, we have
\begin{eqnarray*}
E_{{\rm ISI}}(G_1 \cup \overline{K}_r)&=& E_{{\rm ISI}}(G_1) +E_{{\rm ISI}}(\overline{K}_r)\\
&=& E_{{\rm ISI}}(G_2) +E_{{\rm ISI}}(\overline{K}_r)\\
&=&E_{{\rm ISI}}(G_2 \cup \overline{K}_r).
\end{eqnarray*}
Since $G_1$ and $G_2$ are $S$-noncospectral, therefore $G_1 \cup \overline{K}_r$ and $G_2 \cup \overline{K}_r$ are $S$-noncospectral.
\end{proof}
\begin{corollary}
Suppose $G_1$ and $G_2$ are two $k$-regular $n$-vertex $A$-noncospectral graphs. Then for any $m \geq 2$, $L_m(G_1)\cup \overline{K}_r $ and $L_m(G_1)\cup \overline{K}_r$ are $S$-noncospectral $S$-equienergetic.
\end{corollary}

\end{document}